\documentclass[11pt]{amsart}
\usepackage{latexsym,amscd,amssymb,epsfig} 
\usepackage[dvips]{color}
\usepackage[all]{xy}

\theoremstyle{plain}
  \newtheorem{theorem}{Theorem}[section]
  \newtheorem{proposition}[theorem]{Proposition}
  \newtheorem{lemma}[theorem]{Lemma}
  \newtheorem{corollary}[theorem]{Corollary}
  
  \newtheorem{hypothesis}[theorem]{Hypothesis}
\theoremstyle{definition}
  
  \newtheorem{example}[theorem]{Example}

\theoremstyle{remark}

\numberwithin{equation}{section}

\def\clap#1{\hbox to 5pt{\hss$#1$\hss}}

\def\umapright#1{\smash{
   \mathop{\longrightarrow}\limits^{#1}}}

\def\rmapdown#1{\Big\downarrow\rlap
   {$\vcenter{\hbox{$\scriptstyle#1$}}$}}

\def\tempbaselines
{\baselineskip22pt\lineskip3pt
   \lineskiplimit3pt}
\def\diagram#1{\null\,\vcenter{\tempbaselines
\mathsurround=0pt
    \ialign{\hfil$##$\hfil&&\quad\hfil$##$\hfil\crcr
      \mathstrut\crcr\noalign{\kern-\baselineskip}
  #1\crcr\mathstrut\crcr\noalign{\kern-\baselineskip}}}\,}

\def\pullback#1&#2&#3&#4&#5&#6&#7&#8&{
\diagram{#1&\umapright{#2}&#3\cr
\rmapdown{#4}&&\rmapdown{#5}\cr
#6&\umapright{#7}&#8\cr}}

\def\calC{{\mathcal C}}

\def\calO{{\mathcal O}}

\def\Coker{{\mathop{\rm Coker}\nolimits}}
\def \End{\mathop{\rm End}\nolimits}

\def \Hom{\mathop{\rm Hom}\nolimits}

\def \ker{\mathop{\rm Ker}\nolimits} 
 
\def\Rad{\mathop{\rm Rad}\nolimits}

\def\clap#1{\hbox to 0pt{\hss$#1$\hss}}

\def\QQ{{\mathbb Q}}

\def\ZZ{{\mathbb Z}}

\begin{document}

\title[Grothendieck groups of triangulated categories]{Bilinear forms on Grothendieck groups of triangulated categories}

\author{Peter Webb}
\email{webb@math.umn.edu}
\address{School of Mathematics\\
University of Minnesota\\
Minneapolis, MN 55455, USA}



\begin{abstract}
We extend the theory of bilinear forms on the Green ring of a finite group developed by Benson and Parker to the context of the Grothendieck group of a triangulated category with Auslander-Reiten triangles, taking only relations given by direct sum decompositions. We examine the non-degeneracy of the bilinear form given by dimensions of homomorphisms, and show that the form may be modified to give a Hermitian form for which the standard basis given by indecomposable objects has a dual basis given by Auslander-Reiten triangles. An application is given to the homotopy category of perfect complexes over a symmetric algebra, with a consequence analogous to a result of Erdmann and Kerner.
\end{abstract}

\dedicatory{To Dave Benson on his 60th birthday.}

\maketitle
\section{Introduction}
In their paper \cite{BP}, Benson and Parker introduced a number of new concepts in the theory of the Green ring of a finite group algebra. Among these were a pair of closely related bilinear forms, extending the usual bilinear form on characters of the group. These forms were shown to be non-degenerate and to be related to orthogonality relations between the values of multiplicative functions (called \textit{species}, see also \cite{Web1}), extending the orthogonality relations on characters. Furthermore, the non-degeneracy of the forms was seen to be connected to the existence of Auslander-Reiten sequences (a similar realization was made independently by Auslander~\cite{Aus}).

In this article we copy the start of this theory, in the context of triangulated categories. The Green ring construction may be made for any Krull-Schmidt additive category (ignoring product structure), and the bilinear form given by dimensions of homomorphisms is known to be non-degenerate in many circumstances, by work of Bongartz~\cite{Bon}. We show here the connection with Auslander-Reiten triangles, and obtain a description of the extent to which the bilinear form is non-degenerate in circumstances when Auslander-Reiten triangles exist. Part of this was already anticipated by Benson and Parker, who described the kernel of the bilinear form in the context of the stable module category of for a finite group algebra. Our results extend their theory. This is explained in Section 3, depending on a lemma presented in Section~\ref{basic-lemma}.

The non-degeneracy of their bilinear forms was expressed very nicely by Benson and Parker by the existence of elements in the Green ring dual to the standard basis of indecomposable modules. In the context of a triangulated category the candidate elements constructed from Auslander-Reiten triangles, that we might hope would be dual to the standard basis, do not quite work. To rectify this situation we modify the Green ring by extending it to be a module over the ring of Laurent series $\ZZ[t,t^{-1}]$ where the indeterminate $t$ acts as the shift operator. We have to modify the bilinear form to account for this extension, and our construction is sesquilinear with respect to the automorphism of $\ZZ[t,t^{-1}]$ given by $t\leftrightarrow t^{-1}$. It turns out that, provided we can invert the element $1+t$, there are now dual elements to the indecomposable objects, given by Auslander-Reiten triangles. We explain this in Section~\ref{Hermitian-form-section}.

We conclude by exploring how this theory applies in a particular case: the homotopy category of perfect complexes for a symmetric algebra, the case of the group algebra of a finite group being of special interest. Perfect complexes are finite complexes of finitely generated projective modules. They appear widely in many places in representation theory and elsewhere. Here are three examples that motivate us: the tilting complexes in Rickard's theory \cite{Ric} of derived equivalences, the chain complexes of topological spaces with a free group action (see \cite{AS} for a notable contribution), and the chain complex of the poset of non-identity $p$-subgroups over a $p$-local ring \cite{Web2}.

It is known from work of Wheeler~\cite{Whe} that the Auslander-Reiten quiver components of the category of perfect complexes for a symmetric algebra are of type $\ZZ A_\infty$ (except for blocks of defect zero). With this in mind, we present in Section~\ref{calculation-section} a calculation of the values of our bilinear form on such components, obtaining a result that depends only on the rim of the quiver. We conclude in Section~\ref{bricks-section}  by using the calculation to obtain an analogue for perfect complexes of a theorem of Erdmann and Kerner~\cite{EK} having to do with objects with small endomorphism rings.

\section{The basic lemma}\label{basic-lemma}
\label{AR-lemma}
Let $\calC$ be a triangulated category with the property that indecomposable objects have local endomorphism rings and the Krull-Schmidt theorem holds. 
Our results depend upon the following observation, which follows directly from the definition of an Auslander-Reiten triangle. In the special context of stable module categories of self-injective algebras a very closely related result was proven by Erdmann and Skowro\'nski~\cite[Lemma 3.2]{ES} and used again in the same context in \cite{EK}. The argument for triangulated categories in general appears in \cite[Lemma 2.2]{DPW} and we repeat the short proof for the convenience of the reader.

\begin{lemma}\label{basic-lemma}
Let $X\to Y\to Z\to X[1]$ be an Auslander-Reiten triangle in a Krull-Schmidt triangulated category $\calC$ and let $W$ be an indecomposable object of $\calC$. Consider the long exact sequence obtained by applying $\Hom_\calC(W,-)$ to the triangle.
\begin{enumerate}
\item
If $W\not\cong Z[r]$ for any $r$ then the long exact sequence is a splice of short exact sequences
$$
0\to\Hom(W,X[n])\to\Hom(W,Y[n])\to\Hom(W,Z[n])\to 0.
$$ 
\item
If $W\cong Z[r]$ for some $r$ and $Z\not\cong Z[1]$ the long exact sequence is still the splice of short exact sequences as above, except that the sequences for $n= r$ and $r+1$ combine to give a 6-term exact sequence whose middle connecting homomorphism $\delta$ has rank 1:
$$
\begin{aligned}
0\to&\Hom(W,X[r])\to\Hom(W,Y[r])\to\Hom(W,Z[r])\\
\xrightarrow{\delta}&\Hom(W,X[r+1])\to\Hom(W,Y[r+1])\to\Hom(W,Z[r+1])\\
\to &0.\\
\end{aligned}$$
\item
If $W\cong Z\cong Z[1]$ the long exact sequence becomes a repeating exact sequence with three terms:
$$
\Hom(W,X)\to\Hom(W,Y)\to\Hom(W,Z)\xrightarrow{\delta} \Hom(W,X)\to\cdots
$$
where the connecting homomorphism $\delta$ has rank 1.
\end{enumerate}
The dual result on applying $\Hom_\calC(-,W)$ to the triangle also holds.
\end{lemma}

\begin{proof}
In the long exact sequence
$$
\xymatrix@C=1pc{
&&
\hskip3.7em\clap{\dots}\hskip3.7em\ar[r]^{\beta_{r-1}}&
\Hom(W,Z[r-1])
\ar `r[d] `[l] `[llld] `[dll] [dll]\\
&\hskip3.7em\clap{\Hom(W,X[r])}\hskip3.7em\ar[r]&
\hskip3.7em\clap{\Hom(W,Y[r])}\hskip3.7em \ar[r] ^{\beta_{r}}&
\hskip3.7em\clap{\Hom(W,Z[r])}\hskip3.7em
\ar `r[d] `[l] `[llld] `[dll] [dll]\\
&\Hom(W,X[r+1]) \ar[r]& \Hom(W,Y[r+1]) \ar[r] ^{\beta_{r+1}}& \Hom(W,Z[r+1])
\ar `r[d] `[l] `[llld] `[dll] [dll]\\
&\Hom(W,X[r+2]) \ar[r] &
\hskip3.7em\clap{\cdots}\hskip3.7em & }
$$
the map $\beta_r$ is surjective unless $W\cong Z[r]$ by the lifting property of the Auslander-Reiten triangle. Thus if $W\not\cong Z[r]$ for all $r$, the connecting homomorphisms are all zero, which forces the long sequence to be a splice of short exact sequences. In the case where $W=Z[r]$ for some $r$, the map $\beta_r$ has image $\Rad\End(W)$ and cokernel of dimension 1, again from the lifting property of an Auslander-Reiten triangle. If $Z\not\cong Z[1]$ it follows that $Z[s]\not\cong Z[s+1]$ for any $s$, so $Z[r-1]\not\cong W\not\cong Z[r+1]$ and both $\beta_{r-1}$ and $\beta_{r+1}$ are surjective, giving a six-term exact sequence. When $Z\cong Z[1]$ then all shifts of $Z$ are isomorphic, as are the shifts of $X$ and of $Y$.  Evidently the long exact sequence becomes the three-term sequence shown.
\end{proof}

This lemma has many consequences: one application of it is described in\cite{DPW}. The rest of this paper is devoted to studying its implications for bilinear forms on Grothendieck groups.

\section{A bilinear form on the Grothendieck group}
\label{dim-hom-bilinear-form-section}
From now on we will assume that $\calC$ is a Krull-Schmidt $k$-linear triangulated category, where $k$ is an algebraically closed field, and suppose that $\calC$ is Hom-finite, This means that $\Hom$ spaces between objects are always finite-dimensional.  We define $A(\calC)$ to be the free abelian group with the isomorphism classes $[C]$ of indecomposable objects $C$ as basis. If $C\cong C_1\oplus\cdots\oplus C_n$ we put $[C]=[C_1]+\cdots+ [C_n]$.  We define a bilinear form
$$
\langle\; ,\;\rangle:A(\calC)\times A(\calC)\to \ZZ
$$
by $\langle [C],[D]\rangle:=\dim\Hom_\calC(C,D)$.
If $X\to Y\to Z\to X[1]$ is a triangle in $\calC$ we put $\hat Z :=[Z]+[X]-[Y]$ in $A(\calC)$. In the special context of stable module categories of self-injective algebras the essential features of the following were observed in \cite[3.2]{ES}.

\begin{proposition}\label{AR-triangle-proposition}
Let $X\xrightarrow{\alpha} Y\xrightarrow{\beta} Z\xrightarrow{\gamma} X[1]$ be an Auslander-Reiten triangle in $\calC$ and let $W$ be an indecomposable object. If $Z\not\cong Z[1]$ then
$$
\langle W,\hat Z\rangle=
\begin{cases}
1&\hbox{if }W\cong Z \hbox{ or } Z[-1]\cr
0&\hbox{otherwise.}\cr
\end{cases}
$$
If $Z\cong Z[1]$ then
$$
\langle W,\hat Z\rangle=
\begin{cases}
2&\hbox{if }W\cong Z\cr
0&\hbox{otherwise.}\cr
\end{cases}
$$
\end{proposition}

\begin{proof}
Observe that $\langle W,\hat Z\rangle$ is the alternating sum of the dimensions of the vector spaces in the (not necessarily exact) sequence
$$0\to\Hom(W,X)\xrightarrow{\alpha_*}\Hom(W,Y)\xrightarrow{\beta_*}\Hom(W,Z)\to 0.$$
From Lemma~\ref{basic-lemma} this is a short exact sequence unless $W\cong Z$ or $Z[-1]$, and hence apart from these cases $\langle W,\hat Z\rangle=0$. Assuming that $Z\not\cong Z[1]$, if $W\cong Z$ then the sequence is exact except at the right, where $\Coker\beta_*$ has dimension 1, and if $W\cong Z[-1]$ then the sequence is exact except at the left, where $\ker\alpha_*$ has dimension 1. Thus in these cases $\langle W,\hat Z\rangle=1$. Finally when $W\cong Z\cong Z[1]$ we see from Lemma~\ref{basic-lemma} that both $\ker\alpha_*$ and $\Coker\beta_*$ have dimension 1, so that $\langle W,\hat Z\rangle=2$.
\end{proof}

Let $I$ be the set of shift orbits of isomorphism classes of indecomposable objects in $\calC$, and for each orbit $\calO\in I$ let $A_\calO$ be the span in $A(\calC)$ of the $[M]$ where $M$ belongs to orbit $\calO$. We will now assume that $\calC$ has Auslander-Reiten triangles, by which we mean that every indecomposable object is the start of an Auslander-Reiten triangle, and also the third term in an Auslander-Reiten triangle. This is a strong condition. It holds for the category of perfect complexes $D^b(\Lambda\hbox{-proj})$ when $\Lambda$ is Gorenstein, for cluster categories, and for stable module categories of self-injective algebras, for instance.

\begin{corollary}
\label{non-degeneracy}
Suppose that $\calC$ has Auslander-Reiten triangles. Let
$$
\phi:A(\calC)\to A(\calC)^*:=\Hom_\ZZ(A(\calC),\ZZ)
$$
be the map given by $[W]\mapsto ([X]\mapsto
\langle[W],[X]\rangle)$. Then the image $\phi(A(\calC))$ is a direct sum $\bigoplus_{\calO\in I}\phi(A_\calO)$. Furthermore, if a shift orbit $\calO$ is infinite or has odd length, then the restriction of $\phi$ to $A_\calO$ is injective.
\end{corollary}

\begin{proof}
By Proposition~\ref{AR-triangle-proposition}, as $W$ and $Z$ range through the shifts of some single fixed object the values of $\langle W,\hat Z\rangle$ are the entries of the matrix
$$
\begin{pmatrix}
\ddots&&&&\cr
&1\cr
&1&1\cr
&&1&1\cr
&&&&\ddots\cr
\end{pmatrix}
$$
when the orbit is infinite. The matrix is zero except on the leading diagonal and the diagonal immediately below.  In the case of a finite shift orbit of length $>1$ the matrix is a circulant matrix with the same entries except for a  1 in the top right corner, and the matrix is $(2)$ in the case of a shift orbit of length 1.
If $W$ and $Z$ are not in the same shift orbit then $\langle W,\hat Z\rangle=0$ by Proposition~\ref{AR-triangle-proposition}, so that the matrix of $\phi$ is the direct sum of the matrices just described, one for each shift orbit. The columns of these matrices give the values of the function $\phi(W)$ on the $\hat Z$. This shows, first of all, that $\phi$ sends objects from different shift orbits to independent functions. Furthermore, the matrix indicated has independent columns if either it is infinite, or if it is finite of odd size, so that in these cases the corresponding shift orbit is mapped injectively to $A(\calC)^*$. 
\end{proof}

The last corollary is a statement about the non-degeneracy of the bilinear form we have constructed. We may reword it as follows.  

\begin{corollary}
\label{separate-shift-orbits}
Let $\calC$ be a Krull-Schmidt $k$-linear triangulated category that is Hom-finite and that has Auslander-Reiten triangles. If $W$ is an indecomposable object of $\calC$ then the values of $\dim\Hom(W,Z)$ as $Z$ ranges over indecomposable objects determine the shift orbit to which $W$ belongs. If the shift orbit containing $W$ is either infinite or of odd length, then the isomorphism type of $W$ is determined by the values of $\dim\Hom(W,Z)$.
\end{corollary}

This corollary may be compared with a result of Jensen, Su and Zimmermann \cite[Prop. 4]{JSZ}, who proved something similar without the hypothesis that $\calC$ should have Auslander-Reiten triangles, but with the additional requirement that objects $W$ and $W'$ should have $\Hom(W,W'[n])=0$ for some $n$. Their approach is based on an earlier result of Bongartz and is  appealing because it proceeds by elementary means. In our version of this result we see the connection with Auslander-Reiten triangles, and it applies to categories such as cluster categories, or the stable module category of a group algebra, where the vanishing of a homomorphism space need not hold. 

In the case of the stable module category of a group algebra the phenomena that can occur were analyzed by Benson and Parker~\cite[Theorem 4.4]{BP}. The following is an immediate extension of their theorem to triangulated categories.

\begin{corollary}\label{kernel}
Let $\calC$ be a Krull-Schmidt $k$-linear triangulated category that is Hom-finite and that has Auslander-Reiten triangles. For each even length shift orbit $\calO$ of indecomposable objects of $\calC$, of length $2s$, choose an object $M\in\calO$ and put
$$
\tilde\calO:=  \sum_{i=1}^{2s} (-1)^i [M[i]]
$$
The left kernel of the bilinear form $\langle\; ,\;\rangle$ equals the right kernel, which are both equal to the direct sum $\bigoplus_{\textrm{even length }\calO \in I}\ZZ\tilde\calO$.
\end{corollary}

Note that in this statement $\tilde\calO$ is only defined up to a sign, but this does not make any difference to the conclusion.

\begin{proof}
By Corollary~\ref{non-degeneracy} the left kernel is $\ker\phi=\bigoplus_\calO \ker\phi |_\calO$, and by the matrix description in the proof of that corollary we see that these kernels are as described. Since this description of the left kernel is left-right-symmetric, it is also the right kernel.
\end{proof}

\begin{example}
The following straightforward example illustrates the fact that when shift orbits are finite the bilinear form distinguishes orbits, but not necessarily the objects within an orbit. It is explained by the work of Benson and Parker, as well as the results here.

The stable module category for the ring $k[X]/(X^5)$ (isomorphic to the group algebra of a cyclic group of order 5, when $k$ has characteristic 5) has four indecomposable objects, namely the uniserial modules $V_i=k[X]/(X^i)$ of dimension $i$ where $i=1,2,3,4$. The dimensions of homomorphisms between these objects in the stable category are given in the following table. 
$$
\vbox{{\offinterlineskip
\vbox{
\halign{\hfil$#$\tabskip=.75em plus2em
&\vrule#&\hfil$#$&\hfil$#$&\hfil$#$&\hfil$#$\cr
&&V_1&V_2&V_3&V_4\cr
&height 2pt&&&&\cr
\noalign{\hrule}
&height 2pt&&&&\cr
V_1&&1&1&1&1\cr
&height 2pt&&&&\cr
V_2&&1&2&2&1\cr
&height 2pt&&&&\cr
V_3&&1&2&2&1\cr
&height 2pt&&&&\cr
V_4&&1&1&1&1\cr
}}
}}
$$
Modules $V_1$ and $V_4$ form an orbit of the shift operator, which is the inverse of the syzygy operator, and in this example they cannot be distinguished by dimensions of homomorphisms. The same is true of the modules $V_2$ and $V_3$. We can, however, distinguish the shift orbits by means of dimensions of homomorphisms. The kernel of the bilinear form has basis
$$
\{[V_1]-[V_4], [V_2]-[V_3]\}
$$
according to  Corollary~\ref{kernel}.
\end{example}

\section{A second bilinear form}
\label{Hermitian-form-section}

Benson and Parker show in~\cite{BP}  that the almost split sequences give rise to elements of the Green ring that are dual to the standard basis of indecomposable modules with respect to the dimensions of homomorphisms bilinear form. We have seen in Section~\ref{dim-hom-bilinear-form-section} that a similar statement is not immediately true for triangulated categories with Auslander-Reiten triangles. However, something close to this is true, in that the alternating sum of terms in an Auslander-Reiten triangle has non-zero product with only two indecomposable objects, rather than just one.

We now show how to modify the bilinear form so that Auslander-Reiten triangles do indeed give dual elements to the standard basis. The approach requires us to modify the Grothendieck group as well. We will see also that for the category of perfect complexes over a symmetric algebra we obtain a Hermitian form, and when the algebra is a group algebra the form behaves well with respect to tensor product of complexes.

As before, let $\calC$ be a Krull-Schmidt $k$-linear triangulated category that is Hom-finite, and where $k$ is algebraically closed. We define
$$
A(\calC)^t:=(\ZZ[t,t^{-1}]\otimes_\ZZ A(\calC)) /I
$$
where $t$ is an indeterminate and $I$ is the $\ZZ[t,t^{-1}]$-submodule generated by expressions
$$
1\otimes M[i]-t^i\otimes M
$$
for all objects $M$ in $\calC$ and $i\in\ZZ$. It simplifies the notation to write $M$ instead of $[M]$ at this point in $A(\calC)^t$ and in $A_\QQ(\calC)^t$, so as to avoid the proliferation of square brackets.  We put
$$
A_\QQ(\calC)^t:=\QQ(t)\otimes_{\ZZ[t,t^{-1}]} A(\calC). 
$$
The tensor products are extension of scalars, and with this in mind we will write $t^iM$ instead of $t^i\otimes [M]$. Thus in $A(\calC)^t$ and in $A_\QQ(\calC)^t$ we have $M[i]=t^iM$, so that $t$ acts as the shift operator. As before, write $A_\calO$ for the span of the $M$ in $\calO$, this time in $A(\calC)^t$, regarded as a $\ZZ[t,t^{-1}]$-module via the action $t^iM=M[i]$.

\begin{proposition}
\begin{enumerate}
\item As $\ZZ[t,t^{-1}]$-modules, 
$$
A_\calO\cong
\begin{cases}
 \ZZ[t,t^{-1}]&\hbox{if }\calO\hbox{ is infinite}\cr
  \ZZ[t,t^{-1}]/(t^n-1)&\hbox{if }\calO\hbox{ has size }n\cr
\end{cases}
$$
We have $A(\calC)^t\cong\bigoplus_{{\rm shift\;orbits}\;\calO} A_\calO$.
\item For each shift orbit $\calO$ (taken up to isomorphism) let $M_\calO$ be an object in $\calO$. The $M_\calO$ where $\calO$ is infinite form a basis for $A_\QQ(\calC)^t$ over $\QQ(t)$.
\end{enumerate}
\end{proposition}

\begin{proof}
(1) Evidently $A(\calC)=\bigoplus_{\calO} A_\calO$ and each generator of $I$ lies in some $\ZZ[t,t^{-1}]\otimes_\ZZ A_\calO$. It follows that $A(\calC)^t$ is the direct sum of the images of the $\ZZ[t,t^{-1}]\otimes_\ZZ A_\calO$ in it. Factoring out $I$ has the effect of identifying each basis element $M[i]$ of $A_\calO$ with $t^iM$. The identification of these spaces when the shift operator has infinite or finite order is immediate. 

(2) On tensoring further with $\QQ(t)$ the summands $\ZZ[t,t^{-1}]$ become copies of $\QQ(t)$, and the summands $\ZZ[t,t^{-1}]/(t^n-1)$ become zero.
\end{proof}

On $\ZZ[t,t^{-1}]$ and on $\QQ(t)$ we denote by $\overline{\phantom X}$ the ring automorphism specified by $t\mapsto t^{-1}$. In order to define a new bilinear form on Grothendieck groups we will assume the following hypothesis.

\begin{hypothesis}\label{triangulated-hypothesis}
For every pair of objects $M$ and $N$ in the Krull-Schmidt, Hom-finite, $k$-linear triangulated category $\calC$, we have $\Hom_\calC(M,N[i])\ne 0$ for only finitely many $i$. 
\end{hypothesis}
\noindent This hypothesis is satisfied when $\calC$ is the bounded derived category of finitely generated modules for a finite dimensional $k$-algebra, for instance.
 
 We define a mapping
$$\langle\;,\;\rangle^t: A(\calC)^t\times A(\calC)^t \to \ZZ[t,t^{-1}]
$$
on basis elements $M$ and $N$ by 
$$
\begin{aligned}
\langle M,N\rangle^t:&=\sum_{i\in\ZZ} t^i\dim\Hom_\calC(M,N[i])\\
&=\sum_{i\in\ZZ} t^i\langle M,N[i]\rangle.\\
\end{aligned}
$$
We extend this definition to the whole of $A_\QQ(\calC)^t\times A_\QQ(\calC)^t$ so as to have a sesquilinear form with respect to the ring automorphism $\overline{\phantom{x}}$; that is,
$$
\langle a_1M_1+a_2M_2,N\rangle^t= a_1\langle M_1,N\rangle^t + a_2\langle M_2,N\rangle^t
$$
and 
$$
\langle M,b_1N_1+b_2N_2\rangle^t=\overline{b_1}\langle M,N_1\rangle^t+\overline{b_2}\langle M,N_2\rangle^t
$$
always hold. We also denote the extension of this mapping to $\QQ(t)$ the same way:
$$\langle\;,\;\rangle^t: A_\QQ(\calC)^t\times A_\QQ(\calC)^t \to \QQ(t).
$$

We come to the main result of this section, which establishes the key properties of the form we have just defined, notably that Auslander-Reiten triangles give elements dual to the standard basis of indecomposables, and that for perfect complexes over a symmetric algebra the form is Hermitian.

\begin{theorem}
\label{Hermitian-form}
Assume $\calC$ satisfies Hypothesis~\ref{triangulated-hypothesis}.
\begin{enumerate}
\item The expression defining $\langle\;,\;\rangle^t$  gives a well-defined sesquilinear form on $A(\calC)^t$ and $A_\QQ(\calC)^t$. 
\item Let $X\to Y\to Z\to X[1]$ be an Auslander-Reiten triangle in $\calC$. Define $\hat Z= X+Z-Y$ and let $M$ be an indecomposable object of $\calC$. We have
$$
\langle M,\hat Z\rangle^t = 
\begin{cases}
0&\hbox{unless } M\cong Z[i]\hbox{ for some }i,\cr
(1+t)&\hbox{if }M\cong Z.\cr
\end{cases}
$$
Thus the element ${1\over{1+t}}\hat Z$ is dual on the right to $Z$ in $A_\QQ(\calC)^t$. Similarly
$$
\langle \hat Z, M\rangle^t = 
\begin{cases}
0&\hbox{unless } M\cong X[i]\hbox{ for some }i,\cr
(1+t^{-1})&\hbox{if }M\cong X,\cr
\end{cases}
$$
so that the element ${1\over{1+t^{-1}}}\hat Z$ is dual on the left to $X$ in   $A_\QQ(\calC)^t$.
\item When $\calC=D^b(\Lambda\hbox{-proj})$ is the category of perfect complexes for a symmetric algebra $\Lambda$ the form is Hermitian, in the sense that $\langle M,N\rangle^t=\overline{
\langle N,M\rangle^t}$ always. 
\item When $\calC=D^b(kG\hbox{-proj})$ is the category of perfect complexes for a group algebra $kG$ we have
$$\langle M\otimes_k U,N\rangle^t=\langle M,U^*\otimes_k N\rangle^t
$$
and
$$\langle M,N\rangle^t=\langle M^*,N^*\rangle^t
$$
where $U^*, M^*, N^*$ denote the dual complexes of $kG$-modules.
\end{enumerate}
\end{theorem}

\begin{proof}
(1) The expression that defines the form shows that
$$
\begin{aligned}
\langle M[j],N\rangle^t&= \sum_{i\in\ZZ} t^i\dim\Hom_\calC(M[j],N[i])\\
&=\sum_{i\in\ZZ} t^{i+j}\dim\Hom_\calC(M[j],N[i+j])\\
&=t^j\sum_{i\in\ZZ} t^i\dim\Hom_\calC(M,N[i])\\
&=t^j\langle M,N\rangle^t
\end{aligned}
$$
and by a similar calculation $\langle M,N[j]\rangle^t=t^{-j}\langle M,N\rangle^t$. Thus the form vanishes when elements $M[j]-t^jM$ are put in either side and consequently passes to a well defined sesquilinear form on $A(\calC)^t$ and $A_\QQ(\calC)^t$.

(2) Because of the hypothesis that $\Hom(M,N[i])\ne 0$ for only finitely many $i$ we can never have $M\cong M[1]$ for any $M$. We see by Proposition~\ref{AR-triangle-proposition} that $\langle M,\hat Z\rangle^t$ is only non-zero when $M\cong Z[j]$ for some $j$, and that in the expression for $\langle Z,\hat Z\rangle^t$ the only non-zero terms are $\langle Z,\hat Z\rangle + t\langle Z,\hat Z[1]\rangle =1+t$. The argument for  $\langle \hat Z,N\rangle^t$ is similar.

(3) When $\Lambda$ is symmetric the Nakayama functor is the identity and its left derived functor is the Serre functor on $\calC$, also the identity (see \cite{Hap}). Thus the Serre duality isomorphism on $D^b(\Lambda\hbox{-proj})$ is $\Hom(M,N)\cong \Hom(N,M)^*$. Thus $\langle M,N\rangle^t=\overline{
\langle N,M\rangle^t}$ when $M$ and $N$ are basis elements of $A_\QQ(\calC)^t$, and the same formula follows for arbitrary elements of $A_\QQ(\calC)^t$ by the sesquilinear property of the form.

(4) The formulas follow from the identities
$$
\Hom_{kG}(M\otimes_k U,N)\cong \Hom_{kG}(M,U^*\otimes_k N)
$$
and
$$
\Hom_{kG}(M,N)\cong \Hom_{kG}(N^*,M^*)
$$
for complexes of $kG$-modules.
\end{proof}

It is interesting at this point to compare the bilinear form $\langle M,N\rangle^t$ we have constructed to another bilinear form that appears in \cite[page 13]{Bro}. A bilinear form is defined there as
$$
\langle M,N\rangle:=\sum_{i\in\ZZ} (-1)^i\dim\Hom(M,N[i]),
$$
which is the specialization of $\langle M,N\rangle^t$ on putting $t=-1$. We see from Theorem~\ref{Hermitian-form} part (2) that this is exactly the specialization that destroys the possibility of having dual elements in our sense.

\section{Values of the Hermitian form on Auslander-Reiten quiver components}\label{calculation-section}
\label{form-values}

We present an example to show that the bilinear form $\langle\;,\;\rangle^t$ can be useful in organizing calculations of homomorphism dimensions. The application is to the homotopy category $D^b(\Lambda\hbox{-proj})$ of perfect complexes for a finite dimensional symmetric algebra $\Lambda$ over a field. Perfect complexes are finite complexes of finitely generated projective $\Lambda$-modules. We know from \cite{Hap} that $D^b(\Lambda\hbox{-proj})$ has Auslander-Reiten triangles and that, when $\Lambda$ is symmetric, they have the form $X\to Y\to X[1]\to X[1]$, because the Nakayama functor is the identity.
It was shown by Wheeler \cite{Whe} (see also \cite{HKR}) that, provided $\Lambda$ has no semisimple summand, all components of the Auslander-Reiten quiver of $D^b(\Lambda\hbox{-proj})$ have the form $\ZZ A_\infty$. 

We say that a complex $Z$ lies on the \textit{rim} of the Auslander-Reiten quiver if, in the Auslander-Reiten triangle $X\to Y\to Z\to X[1]$, the complex $Y$ is indecomposable. 
Assuming that $\Lambda$ has no semisimple summand, we will label the objects in a component of the Auslander-Reiten quiver of $D^b(\Lambda\hbox{-proj})$ as shown in Figure~\ref{AR-quiver}.
\begin{figure}[h]
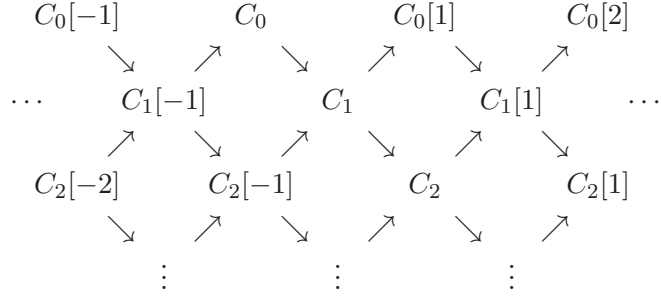

%\hfil
{
\def\tempbaselines
{\baselineskip16pt\lineskip3pt
   \lineskiplimit3pt}
\def\diagram#1{\null\,\vcenter{\tempbaselines
\mathsurround=0pt
    \ialign{\hfil$##$\hfil&&\quad\hfil$##$\hfil\crcr
      \mathstrut\crcr\noalign{\kern-\baselineskip}
  #1\crcr\mathstrut\crcr\noalign{\kern-\baselineskip}}}\,}

\def\clap#1{\hbox to 0pt{\hss$#1$\hss}}

$$\diagram{&\clap{C_0[-1]}&&&&\clap{C_0}&&&&\clap{C_0[1]}&&&&\clap{C_0[2]}\cr
&&\searrow&&\nearrow&&\searrow&&\nearrow&&\searrow&&\nearrow&\cr
\cdots&&&\clap{C_1[-1]}&&&&\clap{C_1}&&&&\clap{C_1[1]}&&&\cdots\cr
&&\nearrow&&\searrow&&\nearrow&&\searrow&&\nearrow&&\searrow&\cr
&\clap{C_2[-2]}&&&&\clap{C_2[-1]}&&&&\clap{C_2}&&&&\clap{C_2[1]}\cr
&&\searrow&&\nearrow&&\searrow&&\nearrow&&\searrow&&\nearrow&\cr
&&&\clap{\vdots}&&&&\clap{\vdots}&&&&\clap{\vdots}\cr
}
$$
}
%\hfil
\caption{Auslander-Reiten quiver component of perfect complexes for a symmetric algebra}\label{AR-quiver}
\end{figure}
Objects on the rim are the shifts of a single object $C_0$, and at distance $n$ from the rim the objects are shifts of an indecomposable $C_n$, which is chosen so that there is a chain of irreducible morphisms $C_0\to C_1\to\cdots\to C_n$.

We will see that the shape of the Auslander-Reiten quiver implies that the values of dimensions of homomorphism spaces are determined entirely by objects on the rim of quiver components, and will calculate the values explicitly from this information. The fact that the bilinear form $\langle\;,\;\rangle^t$ is Hermitian is very useful in organizing the calculation, and we will use the orthogonality of elements determined by Auslander-Reiten triangles in a significant way to give relations between values. To simplify the notation we will write $\sigma_r:=1+t+t^2+\cdots+t^r\in\ZZ[t,t^{-1}]$ when $r\ge 1$, putting $\sigma_0=1$.

\begin{theorem}
\label{values-calculation}
Assume $\Lambda$ is a finite dimensional symmetric $k$-algebra with no semisimple summand, and let $C_0$ and $D_0$ be indecomposable objects in $D^b(\Lambda\hbox{-proj})$ that lie on the rim of their quiver components. Then the values of the form $\langle C_m, D_n\rangle^t$ on objects in the same components as $C_0$ and $D_0$ are entirely determined by knowing $\langle C_0, D_0\rangle^t$. Specifically, if $C_0$ and $D_0$ lie in different quiver components then
$$\langle C_m, D_n\rangle^t=\sigma_m\overline{\sigma}_n \langle C_0, D_0\rangle^t$$
while
$$\langle C_m, C_n\rangle^t=\sigma_m\overline{\sigma}_n \left(\langle C_0, C_0\rangle^t-
{(1+t)(1-t^\mu)\over 1-t^{\mu+1}}\right)
$$
where $\mu$ is the maximum of $m$ and $n$.
\end{theorem}

\begin{proof}
Step 1: we show that if $C_m$ is not a shift of any $D_i$ where $0\le i\le n-1$ then $\langle C_m, D_n\rangle^t=\overline{\sigma}_n \langle C_m, D_0\rangle^t$. We proceed by induction on $n$. The result is true when $n=0$. When $n=1$ the calculation is special because $D_1$ is adjacent to the rim. Since $C_m$ is not a shift of $D_0$ we have
$$
\begin{aligned}
0&=\langle C_m, \hat D_0\rangle^t\\
&= \langle C_m, D_0\rangle^t + \langle C_m, D_0[-1]\rangle^t - \langle C_m, D_1[-1]\rangle^t\\
&=(1+t)\langle C_m, D_0\rangle^t- t\langle C_m, D_1\rangle^t.\\
\end{aligned}
$$
From this we deduce that 
$$
\langle C_m, D_1\rangle^t=t^{-1}(1+t)\langle C_m, D_0\rangle^t = \overline\sigma_1\langle C_m, D_0\rangle^t.
$$
Now suppose that $n\ge 2$ and the result holds for smaller values of $n$. We have
$$
\begin{aligned}
0&=\langle C_m, \hat D_{n-1}\rangle^t\\
&= \langle C_m, D_{n-1}\rangle^t + \langle C_m, D_{n-1}[-1]\rangle^t - \langle C_m, D_n[-1]\rangle^t  - \langle C_m, D_{n-2}\rangle^t\\
&=(1+t)\langle C_m, D_{n-1}\rangle^t- t\langle C_m, D_n\rangle^t - \langle C_m, D_{n-2}\rangle^t.\\
\end{aligned}
$$
This is a recurrence relation for $\langle C_m, D_n\rangle^t$ starting with the values already obtained when $n=0$ and 1, and it is solved by
$$
\langle C_m, D_n\rangle^t=\overline{\sigma}_n \langle C_m, D_0\rangle^t.
$$

Step 2: We deduce that  if $D_n$ is not a shift of any $C_i$ where $0\le i\le m-1$ then $\langle C_m, D_n\rangle^t=\sigma_m \langle C_0, D_n\rangle^t$. This follows from Step 1 on exploiting the fact that the form is Hermitian, for
$$
\langle C_m, D_n\rangle^t=\overline{\langle D_n,C_m\rangle^t}
=\overline{\overline{\sigma_m}\langle D_n,C_0\rangle^t}
=\sigma_m \langle C_0, D_n\rangle^t.
$$

Step 3: We put together Steps 1 and 2 to obtain the first statement of the Proposition, which applies when $C_0$ and $D_0$ lie in different quiver components.

Step 4: We treat the case of two objects in the same quiver component similarly, taking account of the fact that the values $\langle C_m, \hat C_n\rangle^t$ are not always zero. First
$$
\begin{aligned}
1+t&=\langle C_0, \hat C_0\rangle^t\\
&= \langle C_0, C_0\rangle^t + \langle C_0, C_0[-1]\rangle^t - \langle C_0, C_1[-1]\rangle^t\\
&=(1+t)\langle C_0, C_0\rangle^t- t\langle C_0, C_1\rangle^t\\
\end{aligned}
$$
so that
$$
\langle C_0, C_1\rangle^t=t^{-1}(1+t)(\langle C_0, C_0\rangle^t -1)= \overline\sigma_1(\langle C_0, C_0\rangle^t-1)
$$
which agrees with the formula we have to prove. By exactly the same calculation as was used in Step 1, taking $n\ge 2$ and $\hat C_{n-1}=\hat D_{n-1}$ and using the fact that $0=\langle C_0, \hat C_{n-1}\rangle^t$ we obtain the same recurrence
$$0=(1+t)\langle C_0, C_{n-1}\rangle^t- t\langle C_0, C_n\rangle^t - \langle C_0, C_{n-2}\rangle^t$$
valid when $n\ge 2$, and this has solution
$$\langle C_0, C_n\rangle^t=\overline{\sigma}_n \left(\langle C_0, C_0\rangle^t-
{(1+t)(1-t^n)\over 1-t^{n+1}}\right).$$ 
Now if $m\le n$, since $C_m$ is not a shift of any $C_i$ where $0\le i \le n-1$, by Step 2 we deduce 
$$\langle C_m, C_n\rangle^t=\sigma_m\overline{\sigma}_n \left(\langle C_0, C_0\rangle^t-
{(1+t)(1-t^n)\over 1-t^{n+1}}\right).$$
Finally if $m>n$ we use the Hermitian property to deduce
$$
\begin{aligned}
\langle C_m, C_n\rangle^t&=\overline{\langle C_n, C_m\rangle^t}\\
&=
\overline{
\sigma_n\overline{\sigma}_m \left(\langle C_0, C_0\rangle^t-
{(1+t)(1-t^m)\over 1-t^{m+1}}\right)
}
\\
&=
\sigma_m\overline{\sigma}_n \left(\langle C_0, C_0\rangle^t-
{(1+t)(1-t^m)\over 1-t^{m+1}}\right)\\
\end{aligned}
$$
which is what we have to prove.
\end{proof}

\section{Perfect complexes with small endomorphism rings}\label{bricks-section}

We apply the calculation of Section~\ref{calculation-section} to prove a theorem for perfect complexes, analogous to a result of Erdmann and Kerner~\cite{EK} concerning stable module categories of self-injective algebras. They were interested in indecomposable objects with small endomorphism rings, and called a module a \textit{stable brick} if the dimension of its endomorphism ring in the stable category is 1. They showed that if a stable brick occurs in a $\ZZ A_\infty$ quiver component of the stable module category, then all objects in the strip between that object and the rim are also stable bricks.

We start by observing that, in the situation of perfect complexes for a symmetric algebra, there are unfortunately no bricks, other than complexes that are simple projective modules.

\begin{proposition}\label{endomorphism-dimension}
Let $\Lambda$ be a finite dimensional symmetric algebra over a field and let $C$ be an indecomposable perfect complex of $\Lambda$-modules. Then in the derived category, $\dim\End_{D^b(\Lambda)}(C)$ equals 1 if and only if $C$ is a simple projective module concentrated in a single degree, and otherwise $\dim\End_{D^b(\Lambda)}(C)\ge 2$. 
\end{proposition}

Note that the symmetric hypothesis cannot be weakened to self-injective in this proposition, as the complex $\begin{smallmatrix}a\cr b\cr\end{smallmatrix}
\to 
\begin{smallmatrix}b\cr a\cr\end{smallmatrix}$ 
for the Nakayama algebra with projective modules $\begin{smallmatrix}a\cr b\cr\end{smallmatrix},\begin{smallmatrix}b\cr a\cr\end{smallmatrix}$ shows.

\begin{proof}
We may deduce this from the fact that Auslander-Reiten triangles in $D^b(\Lambda\hbox{-proj})$ exist and that the Serre functor is the identity for a symmetric algebra, as described in \cite{Hap} and \cite{Whe}. If $C$ is an indecomposable perfect complex with $\dim\End_{D^b(\Lambda)}(C)=1$, the third morphism in the Auslander-Reiten triangle $C[-1]\to C'\to C\to C$ must be an isomorphism, and so $C'=0$. Wheeler shows in \cite{Whe} that the existence of an irreducible morphism $0\to C$ forces $C$ to be a simple projective concentrated in a single degree.
\end{proof}

In the light of this realization we now consider objects with endomorphism rings of dimension 2, and prove the analogue of the theorem of Erdmann and Kerner.

\begin{theorem}
Let $\Lambda$ be a finite dimensional symmetric algebra over a field and let $C$ be an indecomposable perfect complex of $\Lambda$-modules with $\dim\End_{D^b(\Lambda)}(C)=2$. Then for every complex $D$ in the region of the Auslander-Reiten quiver component between $C$ and the rim, we have $\dim\End_{D^b(\Lambda)}(D)=2$. Furthermore, such a complex $C$ exists at distance $m$ from the rim if and only if there is a perfect complex $C_0$ with $\dim\End_{D^b(\Lambda)}(C_0)=2$ and $\dim\Hom_{D^b(\Lambda)}(C_0,C_0[i])=0$ when $i\ne 0$, $-m\le i\le m$.
\end{theorem}

\begin{proof}
We use the labelling of Figure~\ref{AR-quiver} for the complexes in the quiver component in question and suppose that $\dim\End_{D^b(\Lambda)}(C_m)=2$. From Theorem~\ref{values-calculation} we have
$$\langle C_m, C_m\rangle^t=\sigma_m\overline{\sigma}_m \left(\langle C_0, C_0\rangle^t-
{(1+t)(1-t^m)\over 1-t^{m+1}}\right)
$$
and we suppose that the constant term in this Laurent series is 2. Let us write $\langle C_0, C_0\rangle^t=\sum_i a_it^i$. Then a straightforward calculation shows that the constant term in the expression for $\langle C_m, C_m\rangle^t$ is
$$
a_{-m}+2a_{-m+1}+\cdots+ma_{-1}+(m+1)a_0+ma_1+\cdots+a_m -2m.
$$
We equate this expression to 2 and rearrange, so that
$$
(m+1)a_0+\hbox{non-negative terms}=2(m+1).
$$
Because $a_0\ge 2$ by Proposition~\ref{endomorphism-dimension}, the only way this can happen is if $a_0=2$ and $a_i=0$ if $i\ne 0$, $-m\le i\le m$. Having found these values for the $a_i$ we deduce that the analogous expressions for the constant terms of $\langle C_r, C_r\rangle^t$, $0\le r\le m$ take the value 2 also. 

This proves all of the statements except the final converse: suppose there is a perfect complex $C_0$ with
$
\dim\End_{D^b(\Lambda)}(C_0)=2
$
and
$$
\dim\Hom_{D^b(\Lambda)}(C_0,C_0[i])=0 \hbox{ when } i\ne 0 \hbox{ and  }-m\le i\le m.
$$
If $m=0$ this last condition is vacuous, and by what we have already proved there is a complex on the rim of the required form. If $m\ge 1$ then we may see (in various ways) that $C_0$ itself must lie on the rim. One way is to observe (using similar calculations to what we have already done) that the coefficient of $t$ in $\langle C_0, C_0\rangle^t$ is not 0 when $C_0$ is not on the rim, so that $\Hom_{D^b(\Lambda)}(C_0,C_0[1])\ne 0$, contrary to hypothesis. Knowing that $C_0$  lies on the rim, the expression for the constant term of $\langle C_m, C_m\rangle^t$ now takes the value 2, showing that $C=C_m$ exists as desired.
\end{proof}

\end{document}